\documentclass[11pt]{article}
\usepackage{amsmath,amsthm,amsfonts,amssymb,amscd, amsxtra}

\newcommand{\banacha}{\mathbb X}
\newcommand{\banachb}{\mathbb Y}
\newtheorem{theorem}{Theorem}
\newtheorem{lemma}[theorem]{Lemma}

\newtheorem{corollary}[theorem]{Corollary}
\newtheorem{proposition}[theorem]{Proposition}

\newtheorem{remark}{Remark}

\def\argmin{\operatorname{argmin}}
\begin{document}
\title{Local convergence analysis of a proximal Gauss-Newton method under a majorant condition}

\author{ G. Bouza Allende\thanks{Facultad de Matem\'atica y Computaci\'on, Universidad de la Habana, CU (Email:
{\tt gema@matcom.uh.cu}). The author was partly supported by CAPES.} \and
M. L. N. Gon\c calves\thanks{IME/UFG, Campus II- Caixa
    Postal 131, CEP 74001-970 - Goi\^ania, GO, Brazil (E-mail:{\tt
      maxlng@mat.ufg.br}). The author was partly supported by
CAPES.}  }
 \maketitle
\begin{abstract}
In this paper, the proximal Gauss-Newton method for solving
penalized nonlinear least squares problems is studied.  A local convergence
analysis is obtained under the assumption that
the derivative of the function associated with the penalized least
square problem satisfies a majorant condition. Our analysis provides a clear relationship
between the majorant function and the  function associated with
the penalized least squares problem. The convergence for two
important special cases is also derived.
\end{abstract}


\maketitle
\section{Introduction}\label{sec:int}

We consider the {\it penalized nonlinear least squares} problem
\begin{equation}\label{eq:p}
\min_{x\in \Omega } \;\frac{1}{2}\|F(x)\|^2+J(x),
\end{equation}
where $\banacha$ and $\banachb$ are real or complex Hilbert spaces,
$\Omega\subseteq\banacha$ an open set, $F:\Omega\to \banachb$ is a
continuously differentiable  nonlinear function and $J:\Omega\to
\mathbb{R}\cup\{+\infty\}$ is a proper, convex and  lower
semicontinuous  functional. A wide variety of applications can be found in mathematical programming literature, see
for example \cite{A1,A3,A2}. In particular, if $J(x)=0$, for all $x\in \Omega$, the problem \eqref{eq:p}
becomes the classical nonlinear least squares problem studied in
\cite{MAX1,MAX3, MAX4,MAX5}. In this case,
 a generalization of the  Newton method called
the Gauss-Newton method, can be used. This iterative algorithm computes the sequence
$$
x_{k+1}={x_k}- F'(x_k)^{\dagger}F(x_k), \qquad k=0,1,\ldots,
$$
where $F'(x_k)^{\dagger}$ denotes the Moore-Penrose inverse of the
linear operator $F'(x_k)$.

In this paper,  we  consider the proximal Gauss-Newton
method, introduced in~\cite{A3}, for solving~\eqref{eq:p}. This method extends the classical Gauss-Newton approach. It is defined as
$$
x_{k+1}=\mbox{prox}_{J}^{H(x_k)}(x_k-F'(x_k)^{\dagger}F(x_k)),
\qquad k=0,1,\ldots,
$$
where $\mbox{prox}_{J}^{H(x_k)}$ is the proximity operator
associated to $J$ (see \cite{A4,A5,A6,A3}) with respect to the
metric defined by the operator $H(x_k):=F'(x_k)^*F'(x_k)$. It shall be
 mentioned that the computation  of the proximity operator is
in general not straightforward and it may require an iterative
algorithm itself, since, in general, a closed form is not available.

The aim of this  paper is to present a new local convergence analysis
of proximal Gauss-Newton method under a majorant condition. This
majorant formulation follows the   ideas
used in \cite{MAX1,MAX3,MAX2, MAX4,MAX5}. This  analysis
provides a clear relationship between the majorant function, which
relaxes the Lipschitz continuity of $F'$, and the nonlinear operator
$F$ associated with the penalized nonlinear least squares problem.
Two  majorant functions are also considered. In the first case, which corresponds to functions with  Lipschitz derivative, the classical convergence results are  recovered. The  convergence analysis for analytical operators is discussed for the first time.

The convergence of the sequence generated by the proximal Gauss-Newton method was also studied in~\cite{A3}.
There, instead of majorant function,  the Wang's condition, introduced in \cite{W99,XW10}, is used for the analysis.
In fact, it can be shown that these conditions are equivalent. However,  the formulation as a  majorant condition
is better, due to it provides a clear relationship between the majorant function and the nonlinear function $F$ under consideration. Furthermore, the majorant condition simplifies the  proof of the obtained results.

The organization of the paper is as follows. Next, we list some notations and a basic result used in
our presentation. In Section~\ref{pre},  some results on
Moore-Penrose inverse,  proximity operators and the proximal
Gauss-Newton algorithm are discussed. In Section \ref{lkant}, we state the main
result and, for a better organization of the results, it is divided in three parts. First, some properties of the majorant
function are established. Then in Subsection~\ref{sec:MFNLO}, we present the
relationships between the majorant function and the nonlinear
function $F$.  Finally, in the last part  our main result is proven.
Section \ref{sec:ec} is devoted to show the consequences of this result in particular cases.

\subsection{Notation and auxiliary results} \label{sec:int.1}
The following notations and results are used throughout our
presentation.   Let $\banacha$ and $\banachb$ be Hilbert spaces. The
open and closed balls in $\banacha$  with center $a$ and radius $r$ are
denoted, respectively by
$B(a,r)$ and $ B[a,r].$ For simplicity, given $x \in \banacha$,   we use the short notation
$$\sigma(x):=\|x-x_*\|.$$
From now on, $\Omega\subseteq\banacha$ an open set,
$J:\Omega\to \mathbb{R}\cup\{+\infty\}$ is a proper, convex and
lower semicontinuous  functional and $F:\Omega\to
\banachb$ is a continuously differentiable function such that $F'$ has a closed image
in~$\Omega$.
 We use  $\mathcal L (\banacha,
\banachb)$ to denote the space of bounded linear operators from
$\banacha$ to $\banachb$ and $\mbox{I}_{\banacha}$ corresponds to the
identity operator on $\banacha$. Finally, if $A \in \mathcal L
(\banacha, \banachb)$, then  $Ker(A)$ and $im(A)$ are the kernel and
image of A, respectively, and $A^*$ its adjoint operator.

The following auxiliary results of elementary convex analysis
will be needed:\begin{proposition}\label{pr:conv.aux1} Let
$\epsilon>0$  and  $\tau \in [0,1]$. If $\varphi:[0,\epsilon)
\rightarrow\mathbb{R}$ is convex, then \begin{itemize}
                                        \item The function $l:(0,\epsilon) \to
\mathbb{R}$ defined by
$$
l(t)=\frac{\varphi(t)-\varphi(\tau t)}{t},
$$
is monotone increasing.
\item $
D^+ \varphi(0)={\lim}_{u\to 0+} \; \frac{\varphi(u)-\varphi(0)}{u}
={\inf}_{0<u} \;\frac{\varphi(u)-\varphi(0)}{u}. $
\end{itemize}
\end{proposition}
\begin{proof} See Theorem 4.1.1 and Remark 4.1.2,  pp. 21 of  \cite{HL93}.
\end{proof}
\section{Preliminary}\label{pre}
In this section  some results on  Moore-Penrose
inverse and proximity operators will be presented. Then, the
algorithm to solve problem~\eqref{eq:p} and  some properties
related to it will be introduced.
\subsection{Generalized inverses}
In this section some results on  Moore-Penrose
inverse, will be presented. More details can be found in \cite{G1,W1}.

Let $A\in \mathcal L (\banacha, \banachb) $  with a closed image. The
Moore-Penrose inverse of $A$
 is the linear operator $A^\dagger \in \mathcal L (\banachb, \banacha)$  which satisfies:
$$AA^{\dagger}A=A, \quad A^{\dagger}AA^{\dagger}=A^{\dagger}, \quad (AA^{\dagger})^{*}=AA^{\dagger}, \quad (A^{\dagger}A)^{*}=A^{\dagger}A. $$
From the definition of the Moore-Penrose inverse, it is easy to see
that
\begin{equation}\label{1111}
A^{\dagger}A= \mbox{I}_{\banacha}-\Pi_{Ker(A)}, \qquad AA^{\dagger}=
\Pi_{im(A)},
\end{equation}
 where $\Pi_{E}$ denotes the projection of $\banacha$ onto subspace E.

 If $A$ is injective, then
 \begin{equation}\label{121212}
A^{\dagger}=(A^*A)^{-1}A^{*}, \qquad A^{\dagger}A=I_{\banacha}, \qquad \|A^{\dagger}\|^2=\|(A^*A)^{-1}\|.
\end{equation}
We end this part with a result concerning the variation of the pseudo-inverse, see \cite{A3,G1,W1}.
\begin{lemma} \label{lem:ban}
Let $A, B \in \mathcal L (\banacha, \banachb) $ with closed images.
If $A$ is injective and $\|A^\dagger\|\|A-B\|<1$, then
 $B$ is injective and
$$
\|B^\dagger\|\leq \frac{\|A^\dagger\|}{ 1- \|A^\dagger\|\|A-B\|},
\qquad \|B^\dagger-A^\dagger\|\leq
\frac{\sqrt{2}\|A^\dagger\|^2\|A-B\|}{ 1- \|A^\dagger\|\|A-B\|}.
$$
\end{lemma}
\subsection{Proximity operators}
Proximity operators  were introduced by Moreau and their use in
signal theory goes back to \cite{A4}. We briefly recall some essential facts below and
refer the reader to \cite{A4,A5,A3} for more details.

Let $H:\banacha\to \banacha$ be a  continuously, positive and
selfajoint, bounded from below and, therefore, invertible operator.
Then we have a new scalar product on $\banacha$ by setting $\langle
x,z \rangle_{H}=\langle x,Hz \rangle$. Hence, the corresponding
induced norm $\|.\|_H$ is equivalent to the given norm on
$\banacha$, since the following inequalities hold
$$
 \frac{1}{\|H^{-1}\|}\|x\|^2\leq\|x\|^2_H\leq \|H\|\|x\|^2.
$$
The {\it Moreau-Yosida aproximation} of $J$  with respect to the scalar product
induced by $H$ is the functional $M_j: \banacha \to \mathbb{R}$
defined by setting
\begin{equation}\label{a11}
M_J(z)=\inf_{x\in
\banacha}\left\{J(x)+\frac{1}{2}\|x-z\|^2_H\right\}.
\end{equation}
Recalling that $J$ is a convex, lower
semicontinuous and proper function $J:\banacha \to
\mathbb{R}\cup\{+\infty\}$, it is easy to prove that the infimum of the last equation is attained at a
unique point. Therefore, let us call $\mbox{prox}_{J}^{H}(z)$, the {\it
proximity operator} associated to $J$ and $H$
\begin{equation} \label{pro}
  \begin{array}{rcl}
  \mbox{prox}_{J}^{H}:\banacha & \to &\banacha\\
    z&\mapsto& M_J(z)=\argmin_{x \in \banacha}\left\{J(x)+\frac{1}{2}\|x-z\|^2_H\right\}.
  \end{array}
\end{equation}
Writing the first order
optimality conditions for \eqref{a11}, we obtain that
$$
p=\mbox{prox}_{J}^{H}(z)\leftrightarrow 0\in \partial J(p)+H(p-z)\leftrightarrow Hz\in (\partial J+H)(p),
$$
which using that the minimum in \eqref{a11} is attained at a unique
point leads to
$$\mbox{prox}_{J}^{H}(z)=(\partial J+H)^{-1}(Hz).$$
This part ends with an important property of proximity operator.
\begin{lemma}\label{lemma4}
Let $H_1$ and $H_2$ be two continuous positive selfadjoint operators
on $\banacha$, both bounded from bellow. Then,
$$
\|\mbox{prox}_{J}^{H_1}(z_1)-\mbox{prox}_{J}^{H_2}(z_2)\|\leq  \sqrt{\|H_1\|\|H_1^{-1}\|}\|z_1-z_2\|+\|H^{-1}_1\|\|(H_1-H_2)(z_2-\mbox{prox}_{J}^{H_2}(z_2)\|,
$$
for every  $ z_1, z_2 \in \banacha$.
\end{lemma}
\begin{proof}
 See Remark~4 in \cite{A3}.
 \end{proof}
\subsection{The proximal Gauss-Newton method}
In this section we present the algorithm to solve \eqref{eq:p} as
well as  some related properties.

The goal of this method,  introduced in \cite{A3}, is to find stationary points of
problem \eqref{eq:p}  as follows:
\begin{equation}\label{aa1}
x_{k+1}=\mbox{prox}_{J}^{H(x_k)}(x_k-F'(x_k)^{\dagger}F(x_k)),
\qquad k=0,1,\ldots,
\end{equation}
where $H(x_k)=F'(x_k)^*F'(x_k)$ and $\mbox{prox}_{J}^{H(x_k)}$ is
the proximity operator associated to $J$ and $H(x_k)$ as defined in
\eqref{pro}.
\begin{remark}
As proved in Propositon~6 of \cite{A3}, given $x_k\in
\banacha$, if $F'(x_k)$ is injective with closed image, then
$x_{k+1}$ satisfies
$$x_{k+1}=\mbox{arg}\min_{x \in \banacha} \frac{1}{2}\|F(x_n)+F'(x_k)(x-x_k)\|^2+J(x).$$
This problem  can be solved using first order methods for the minimization
of nonsmooth convex functions, such as bundle methods or
forward-backward methods (see \cite{Com1,HL94}). We will use the proximal point
formulation because the theoretical results of this area will be very useful
for the proof of the convergence of the
method.
\end{remark}

In the following, we establish the connection between the stationary point of the function defined in~\eqref{eq:p} and the fixed points of proximal point operator.

\begin{proposition}\label{pro:a1}
Let $x_*\in \Omega$ such that  $-F'(x_*)^*F(x_*)\in \partial J(x_*)$,
i.e., $x_*$ satisfies the first order conditions for local
minimizers of \eqref{eq:p}. Assume that $F'(x_*)$ is injective and
$im(F'(x_*))$ is closed, then $x_*$ satisfies the fixed point
equation
$$x_*=\mbox{prox}_{J}^{H(x_*)}(x_*-F'(x_*)^{\dagger}F(x_*)),$$
where $H(x_*)=F'(x_*)^*F'(x_*)$.
\end{proposition}
\begin{proof}
The proof follows the same ideas of the proof of Proposition~5 in
\cite{A3}.
 \end{proof}
\section{Local analysis  for the Gauss-Newton method } \label{lkant}
Our goal is to state and prove a local theorem for the proximal
Gauss-Newton method defined in \eqref{aa1}. First, we  show some
results regarding the scalar majorant function, which relaxes the
Lipschitz condition to $F'$. Then, we establish the main
relationships between the majorant function and the nonlinear
function $F$.  Finally we  obtain that the Gauss-Newton method is well-defined and converges.
 The statement of the theorem is as follows:
\begin{theorem}\label{th:nt}
Let $\Omega\subseteq \banacha$ be an open set,
$J:\Omega\to
\mathbb{R}\cup\{+\infty\}$ a proper, convex and lower semicontinuous
functional and $F:{\Omega}\to
\banachb$ a continuously differentiable function such that $F'$  has a closed image in $\Omega$. Let $x_* \in \Omega,$ $R>0$ and
$$
c:=\|F(x_*)\|, \qquad \beta:=\|F'(x_*)^{\dagger}\|,\qquad \kappa:= \beta\left\|F'(x_*)\right\|\qquad \delta:=\sup \left\{ t\in [0, R): B(x_*, t)\subset\Omega \right\}.
$$
Suppose that $-F'(x_*)^*F(x_*)\in \partial J(x_*)$,
$F '(x_*)$ is injective
and there exists a
continuously differentiable function $f:[0,\; R)\to \mathbb{R}$ such that
\begin{equation}\label{Hyp:MH}
\beta\left\|F'(x)-F'(x_*+\tau (x-x_*))\right\| \leq
f'\left(\sigma(x)\right)-f'\left(\tau\sigma(x)\right),
\end{equation}
where $x\in B(x_*, \delta)$, $ \tau \in [0,1]$ and $\sigma(x)=\|x-x_*\|$, and
\begin{itemize}
  \item[{\bf h1)}]  $f(0)=0$ and $f'(0)=-1$;
  \item[{\bf  h2)}]  $f'$ is convex and  strictly increasing;
   \item[{\bf  h3)}]  $[(1+\sqrt{2})\kappa+1] \,c\, \beta  D^+ f'(0)<1$.
\end{itemize}
Let be given the positive constants $\nu  :=\sup \left\{t \in[0, R):f'(t)<0\right\},$
\[
\rho :=\sup \left\{t \in(0, \nu):\frac{\left[f'(t)+1+\kappa\right]\left[tf'(t)-f(t)+c
 \beta(1+\sqrt{2})(f'(t)+1)\right]+c
 \beta\left[f'(t)+1\right]}{t[f'(t)]^2}<1\right\},
\]
$$ r :=\min  \left\{\rho, \,\delta \right\}.$$
Given $H(x_k)=F'(x_k)^*F'(x_k)$, define  $\mbox{prox}_{J}^{H(x_k)}$ as the
 proximity operator associated to $J$ and $H(x_k)$, see
\eqref{pro}.
Then, the proximal Gauss-Newton method for solving \eqref{eq:p}, with
starting point $x_0\in B(x_*, r)/\{x_*\}$
\begin{equation} \label{eq:DNS}
x_{k+1}=\mbox{prox}_{J}^{H(x_k)}\big({x_k}-F'(x_k)^{\dagger}F(x_k)\big),\qquad
k=0,1\ldots,
\end{equation}
 is well defined, the generated sequence $\{x_k\}$ is
contained in $B(x_*,r)$, converges to $x_*$ and
\begin{multline} \label{eq:q2}
    \|x_{k+1}-x_*\| \leq
    \frac{[f'(\sigma(x_0))+1+\kappa][f'(\sigma(x_0))\sigma(x_0)-f(\sigma(x_0))]}{[\sigma(x_0)f'(\sigma(x_0))]^2}{\|x_k-x_*\|}^{2}+\\ \frac{ (1+\sqrt{2})\beta c [f'(\sigma(x_0))+1]^2}{[\sigma(x_0)f'(\sigma(x_0))]^2}{\|x_k-x_*\|}^{2}+
    \frac{c\beta[(1+\sqrt{2})\kappa+1][f'(\sigma(x_0))+1]}{\sigma(x_0)[f'(\sigma(x_0))]^2}\|x_k-x_*\|,
\end{multline}
 for all $k=0,1,\ldots.$
\end{theorem}
\begin{remark}
If $J=0$, the proximal Gauss-Newton method becomes the classical
Gauss-Newton method. However, with respect to the  radius of the convergence ball
this result does not correspond with the classical approach, see  Theorem~7
of \cite{MAX3}. The reason is that the upper bound given  in
Lemma~\ref{lemma4} is not affected if  $J=0$.
\end{remark}
\begin{remark}
If the inequality in \eqref{Hyp:MH} holds only for $\tau=0$,   an analogous theorem is true.
In fact, if   the definition of $\rho$ is replaced
{\small
 $$
\rho=\sup\left\{t\in(0,\nu):\frac{[f'(t)+1+\kappa][tf'(t)+f(t)+2t+c
 \beta(1+\sqrt{2})(f'(t)+1)]+c
 \beta[f'(t)+1]}{t[f'(t)]^2}<1\right\},
 $$
}
the well definition of proximity operator, the inclusion of the computed sequence in $B(x^*,r)$ and  its convergence are guaranteed.  In particular:
\begin{multline*}
    \|x_{k+1}-x_*\| \leq
    \frac{[f'(\sigma(x_0))+1+\kappa][f'(\sigma(x_0))\sigma(x_0)+f(\sigma(x_0))+2\sigma(x_0)]}{[\sigma(x_0)f'(\sigma(x_0))]^2}{\|x_k-x_*\|}^{2}+\\ \frac{ (1+\sqrt{2})\beta c [f'(\sigma(x_0))+1]^2}{[\sigma(x_0)f'(\sigma(x_0))]^2}{\|x_k-x_*\|}^{2}+
    \frac{c\beta[(1+\sqrt{2})\kappa+1][f'(\sigma(x_0))+1]}{\sigma(x_0)[f'(\sigma(x_0))]^2}\|x_k-x_*\|,
\end{multline*}
 for all $k=0,1,\ldots.$
\end{remark}
As before,  $H(x_k)=F'(x_k)^*F'(x_k)$ and  $\mbox{prox}_{J}^{H(x_k)}$ is
the proximity operator defined in
\eqref{pro}.  For the zero-residual problems, i.e., $c=0$,  Theorem \ref{th:nt} becomes:
\begin{corollary} \label{col:pc1}
Let $\Omega\subseteq \banacha$ be an open set,
$J:\Omega\to
\mathbb{R}\cup\{+\infty\}$ a proper, convex and lower semicontinuous
functional and $F:{\Omega}\to
\banachb$ a continuously differentiable function such that $F'$  has a closed image in $\Omega$. Let $x_* \in \Omega,$ $R>0$ and
$$
\beta:=\|F'(x_*)^{\dagger}\|,\qquad \kappa:= \beta\left\|F'(x_*)\right\|\qquad \delta:=\sup \left\{ t\in [0, R): B(x_*, t)\subset\Omega \right\}.
$$
Suppose that  $F(x_*)=0$, $0 \in \partial J(x_*)$,
$F '(x_*)$ is injective
and there exists a
continuously differentiable function $f:[0,\; R)\to \mathbb{R}$ such that
$$
\beta\left\|F'(x)-F'(x_*+\tau (x-x_*))\right\| \leq
f'\left(\sigma(x)\right)-f'\left(\tau\sigma(x)\right),
$$
where $x\in B(x_*, \delta)$, $ \tau \in [0,1]$ and $\sigma(x)=\|x-x_*\|$, and
\begin{itemize}
  \item[{\bf h1)}]  $f(0)=0$ and $f'(0)=-1$;
  \item[{\bf  h2)}]  $f'$ is convex and  strictly increasing.
   \end{itemize}
Let be given the positive constants $\nu  :=\sup \left\{t \in[0, R):f'(t)<0\right\},$
\[
\rho :=\sup \left\{t \in(0, \nu):\frac{\left[f'(t)+1+\kappa\right]\left[tf'(t)-f(t)\right]}{t[f'(t)]^2}<1\right\},
\qquad r :=\min  \left\{\rho, \,\delta \right\}.
\]
Then, the proximal Gauss-Newton method for solving \eqref{eq:p}, with
starting point $x_0\in B(x_*, r)/\{x_*\}$
$$
x_{k+1}=\mbox{prox}_{J}^{H(x_k)}\big({x_k}-F'(x_k)^{\dagger}F(x_k)\big),\qquad
k=0,1\ldots,
$$
 is well defined, the generated sequence $\{x_k\}$ is
contained in $B(x_*,r)$, converges to $x_*$ and
$$
    \|x_{k+1}-x_*\| \leq
    \frac{[f'(\sigma(x_0))+1+\kappa][f'(\sigma(x_0))\sigma(x_0)-f(\sigma(x_0))]}{[\sigma(x_0)f'(\sigma(x_0))]^2}{\|x_k-x_*\|}^{2}, \qquad k=0,1,\ldots.
$$
\end{corollary}
In order to prove Theorem \ref{th:nt} we need some results. From
now on, we assume that all the assumptions of Theorem \ref{th:nt}
hold.
\subsection{The majorant function } \label{sec:PMF}
Our first goal is to show that the constant $\delta$ associated
with $\Omega$ and the constants  $\nu$ and $\rho$
associated with the majorant function~$f$ are positive. Also, we
will prove some results related to the function~$f$.

We begin by noting that $\delta>0$, because $\Omega$ is an open
set and $x_*\in \Omega$.
\begin{proposition}  \label{pr:incr1}
The constant $ \nu $ is positive and $f'(t)<0$ for all $t\in (0, \nu).$
\end{proposition}
\begin{proof}
As $f'$ is continuous in $(0,R)$ and $f'(0)=-1,$ there exists
$\epsilon>0$ such that $f'(t)<0$ for all $t\in (0, \epsilon)$. Hence,
$\nu>0.$ Now, using {\bf  h2} and definition of $\nu$ the last part
of the proposition follows.
\end{proof}
\begin{proposition} \label{pr:incr101}
The following functions are positive and increasing:
\begin{itemize}
 \item[{\bf i)}] $[0,\,\nu ) \ni t \mapsto -1/f'(t);$
 \item[{\bf ii)}]  $[0,\, \nu) \ni t \mapsto -[f'(t)+1+\kappa]/f'(t);$
 \item[{\bf iii)}]  $(0,\, \nu) \ni t \mapsto [tf'(t)-f(t)]/t^2;$
\item[{\bf iv)}]  $(0,\, \nu) \ni t \mapsto [f'(t)+1]/t.$
\end{itemize}
As a consequence,
$$
(0,\, \nu)\ni t\mapsto \frac{[f'(t)+1+\kappa][tf'(t)-f(t)]}{[tf'(t)]^2},
\qquad (0,\, \nu)\ni t\mapsto
\frac{[f'(t)+1]^2}{[tf'(t)]^2},  \qquad (0,\, \nu)\ni t\mapsto
\frac{f'(t)+1}{t[f'(t)]^2},
$$  are also  positive and increasing  functions.
\end{proposition}
\begin{proof}
Items~{\bf i} and {\bf ii} are immediate, because ${\bf h1}$, ${\bf h2}$ and Proposition~\ref{pr:incr1}  imply that  $f'$ is strictly increasing  and $-1\leq f'(t)<0$ for all $t\in [0, \nu).$

Now, note that after some simple algebraic manipulations we have
  $$
  \frac{tf'(t)-f(t)}{t^2}=\int_{0}^{1}\frac{f'(t)-f'(\tau t)}{t}\; d\tau.
  $$
Hence, as $f'$ is strictly increasing $({\bf h2})$, we obtain that the function of  item~{\bf iii} is positive. Moreover,
   combining the last equation and  Proposition~\ref{pr:conv.aux1}  with $f'=\varphi$ and $\epsilon=\nu$, we conclude function of item~{\bf iii} is increasing.  So, item~{\bf iii} is proved.

Assumption $\bf h1$ and $\bf h2$ imply that the function of item~{\bf iv} is positive. Hence, to conclude item~{\bf iv} use  ${\bf h2}$, $f'(0)=-1$ and
Proposition~\ref{pr:conv.aux1} with $f'=\varphi,$ $\epsilon=\nu$ and
$\tau =0.$

To prove that the functions in the last part are positive and increasing combine items {\bf i}, {\bf ii} and {\bf iii} for the first function and  items {\bf i} and {\bf iv} for the second and third functions.
\end{proof}
\begin{proposition}\label{pr:incr102}
The constant $\rho $ is positive and there holds
$$
\frac{\left[f'(t)+1+\kappa\right]\left[tf'(t)-f(t)+c
 \beta(1+\sqrt{2})(f'(t)+1)\right]+c
 \beta\left[f'(t)+1\right]}{t[f'(t)]^2}<1, \qquad \forall \; t\in
(0, \,\rho).
$$
\end{proposition}
\begin{proof}
First, using $\bf h1$ and some algebraic manipulation gives
$$
\frac{tf'(t)-f(t)}{t}=\left[f'(t)-\displaystyle\frac{f(t)-f(0)}{t-0}\right], \qquad  \frac{f'(t)+1}{t}=\frac{f'(t)-f'(0)}{t-0}.
$$
Since $f'(0)=-1$ and $f'$ is convex,   last equations and  Proposition~\ref{pr:conv.aux1} lead  to
$$
\lim_{t \to 0}{[tf'(t)-f(t)]}/{t}=0,  \qquad \lim_{t \to 0} {[f'(t)+1]}/{t}=D^+f'(0),$$
which, combined with  $f'(0)=-1$ and simples calculus yields
\begin{align*}
&\lim_{t \to 0}
\frac{\left[f'(t)+1+\kappa\right]\left[tf'(t)-f(t)+c
 \beta(1+\sqrt{2})(f'(t)+1)\right]+c
 \beta\left[f'(t)+1\right]}{t[f'(t)]^2}\\
&=\kappa c\beta(1+\sqrt{2})D^+f'(0)+c\beta D^+f'(0)=c\beta[(1+\sqrt{2})\kappa+1]D^+f'(0).
\end{align*}
Now, using {\bf h3}, i.e.,  $[(1+\sqrt{2})\kappa+1]c\beta D^+f'(0)<1$, we conclude
that there exists a $\epsilon>0$ such that
$$
\frac{\left[f'(t)+1+\kappa\right]\left[tf'(t)-f(t)+c
 \beta(1+\sqrt{2})(f'(t)+1)\right]+c
 \beta\left[f'(t)+1\right]}{t[f'(t)]^2}<1, \qquad
 t\in
(0, \epsilon),
$$
So, $\epsilon\leq\rho$, which proves the first
statement.

To conclude the proof, we use the definition of $\rho$ and  Proposition~\ref{pr:incr101}.
\end{proof}
\subsection{Relationship of the majorant function with the non-linear function $F$} \label{sec:MFNLO}
In this part we will present the main relationships between
the majorant function $f$ and the function $F$ associated with the  problem \eqref{eq:p}. As usual
$\sigma(x)=\|x-x_*\|$.
\begin{lemma} \label{wdns}
Let $x \in \Omega$. If \,$\sigma(x)<\min\{\nu,\delta\}$, then
$H(x)=F'(x)^* F'(x) $ is invertible and the following
inequalities hold
$$
\|F'(x)^{\dagger}\|\leq \frac{-\beta}{f'(\sigma(x))}, \qquad
 \|F'(x)^{\dagger}-F'(x_*)^{\dagger}\|
<
\frac{-\sqrt{2}\beta[f'(\sigma(x))+1]}{f'(\sigma(x))}.
 $$
In particular, $H(x)=F'(x)^* F'(x)$ is invertible in $B(x_*, r)$.
\end{lemma}
\begin{proof}
Let $x \in \Omega$ such that \,$\sigma(x)<\min\{\nu,\delta\}$. Since $\sigma(x)<\nu$, by Proposition~\ref{pr:incr1}, $f'(\sigma(x))<0$. Using the definition of $\beta$, the inequality \eqref{Hyp:MH} and {\bf h1} we have
$$
\|F'(x_*)^{\dagger}\|\|F'(x)-F'(x_*)\|=\beta\|F'(x)-F'(x_*)\|\leq
f'(\sigma(x))-f'(0)=f'(\sigma(x))+1 < 1.
$$
Taking account that $F'(x_*)$  is injective, in view of Lemma~\ref{lem:ban},
 $F'(x)$ is injective. So, $H(x)$ is invertible. Moreover, again by Lemma~\ref{lem:ban}
 $$\|F'(x)^\dagger\|\leq \frac{\|F'(x_*)^\dagger\|}{1-\|F'(x_*)^\dagger\|\|F'(x_*)-F'(x)\|}, \quad \|F'(x_*)^\dagger-F'(x)^\dagger\|\leq \frac{\sqrt{2}\|F'(x_*)^\dagger\|^2\|F'(x_*)-F'(x)\|}{1-\|F'(x_*)^\dagger\|\|F'(x_*)-F'(x)\|}.$$
 Now, using the definition of $\beta$, inequality   \eqref{Hyp:MH}, {\bf h1} and that $\sigma(x)<\nu$,  we have
$$\frac{1}{1-\|F'(x_*)^{\dagger}\|\|F'(x)-F'(x_*)\|}\leq
\frac{1}{1-(f'(\sigma(x))-f'(0))}=\frac{-1}{f'(\sigma(x))}.
$$
Thus, combining the last  inequalities,  we obtain the desired bounds for $\|F'(x)^\dagger\|$ and \linebreak $\|F'(x)^{\dagger}-F'(x_*)^{\dagger}\|$.
The last part  follows by noting that $r\leq\min\{\nu,\delta\}$.
\end{proof}

To prove the convergence of sequence $\{x_k\}$ on Theorem~\ref{th:nt},  the following relations will be needed.
\begin{lemma}  \label{pr:H}
Let $x \in \Omega$. If  $\sigma(x)<\min\{\nu,\delta\}$, then
\begin{itemize}
 \item[{\bf i)}]  $\|H(x)\|^{1/2}\leq [{f'(\sigma(x))+1+\kappa}]/\beta$;
 \item[{\bf ii)}]  $\|H(x)^{-1}\|^{1/2}\leq {-\beta}/{[f'(\sigma(x))]}$;
 \item[{\bf iii)}]  $\beta\|(H(x)-H(x_*))F'(x_*)^{\dagger}\|\leq (f'(\sigma(x))+2+\kappa)(f'(\sigma(x))+1).$
\end{itemize}
\end{lemma}
\begin{proof}
First, simple calculus, inequality in \eqref{Hyp:MH} and definitons of $\beta$ and $\kappa$ gives
\begin{equation}\label{a2012}
\beta\|F'(x)\|=\|F'(x_*)^{\dagger}\|\|F'(x)\|\leq \beta\|F'(x)-F'(x_*)\|+\beta\|F'(x_*)\|\leq
f'(\sigma(x))+1+\kappa.
 \end{equation}
As
$\|H(x)\|^{1/2}=\|F'(x)^*F'(x)\|^{1/2}=\|F'(x)\|$, the first
statement follows.

Now, to show item {\bf ii}, use definition $H$, last inequality
in \eqref{121212} and Lemma~\ref{wdns}.

For {\bf iii},  note that the definition
$H$,  some algebraic manipulations and \eqref{1111} gives
\begin{align*}
\beta\|(H(x)-H(x_*))F'(x_*)^{\dagger}\|&=\beta\|F'(x)^*(F'(x)-F'(x_*))F'(x_*)^{\dagger}+(F'(x)-F'(x_*))^*
\Pi_{im(F'(x_*))}\|\\
&\leq(\|F'(x)\|\|F'(x_*)^{\dagger}\|+1)\beta\|F'(x)-F'(x_*)\|,
\end{align*}
which, combined with   \eqref{a2012} and  inequality in \eqref{Hyp:MH}
imply the desired statement.
\end{proof}
\begin{remark}
Note that in the Lemmas~\ref{wdns} and \ref{pr:H}, we have used the fact that condition \eqref{Hyp:MH} holds only for $\tau=0$.
\end{remark}
It is convenient to study the linearization error of $F$ at a point in~$\Omega$, for which we define
\begin{equation}\label{eq:def.er}
  E_F(x,y):= F(y)-\left[ F(x)+F'(x)(y-x)\right],\qquad y,\, x\in \Omega.
\end{equation}
We will bound this error by the error in the linearization on the
majorant function $f$
\begin{equation}\label{eq:def.erf}
        e_f(t,u):= f(u)-\left[ f(t)+f'(t)(u-t)\right],\qquad t,\,u \in [0,R).
\end{equation}
\begin{lemma}  \label{pr:taylor}
Let $x \in \Omega$. If  $\sigma(x)< \delta$, then  $ \beta\|E_F(x, x_*)\|\leq
e_f(\sigma(x), 0). $
\end{lemma}
\begin{proof}
 Since  $B(x_*, \delta)$ is convex,  we obtain that $x_*+\tau(x-x_*)\in B(x_*, \delta)$, for $0\leq \tau \leq 1$.
 Thus, as $F$ is continuously differentiable in $\Omega$, the definition of $E_F$ and some simple manipulations yield
$$
\beta\|E_F(x,x_*)\|\leq \int_0 ^1 \beta\left\|
[F'(x)-F'(x_*+\tau(x-x_*))]\right\|\,\left\|x_*-x\right\| \;
d\tau.
$$
From the last inequality and the assumption \eqref{Hyp:MH}, we
obtain
$$
\beta\|E_F(x,x_*)\| \leq \int_0 ^1
\left[f'\left(\sigma(x)\right)-f'\left(\tau\sigma(x)\right)\right]\sigma(x)\;d\tau.
$$
Evaluating the above integral and using the definition of $e_f$, the
statement follows.
\end{proof}
\begin{remark}
If the inequality in \eqref{Hyp:MH} holds only for $\tau=0$, then the upper  bound of $\beta\|E_F(x, x_*)\|$ in the previous Lemma becomes
$ e_f(\sigma(x), 0)+ 2(f(\sigma(x))+\sigma(x)) . $
\end{remark}

In particular, Lemma \ref{wdns} guarantees  that  $H(x)=F'(x)^*F'(x)$
is invertible in $B(x_*, r)$ and consequently, $F'(x)^{\dagger}$ and $\mbox{prox}_{J}^{H(x)}$ are well defined in this region. Hence, the  proximal Gauss-Newton iteration map is also well defined.  Let us call $\mathcal{G}_{F}$,  the
 proximal Gauss-Newton iteration map for $F$ in that region:
\begin{equation} \label{NFS}
  \begin{array}{rcl}
\mathcal{G}_{F}:B(x_*, r) &\to& \banacha\\
    x&\mapsto& \mbox{prox}_{J}^{H(x)}(G_{F}(x))
  \end{array},
\end{equation}
where,
\begin{equation}\label{NF}
H(x)=F'(x)^{*}F'(x), \qquad G_{F}(x)=x- F'(x)^{\dagger}F(x).
\end{equation}
Take  $x\in
B(x_*, r)$. Note that
the point computed by the proximal Gauss-Newton iteration, $\mathcal{G}_{F}(x)$, may not be an element of $B(x_*,
r)$, or may not even belong to the domain of $F$.   To
ensure that the Gauss-Newton iterations may be repeated indefinitely, this is
enough to guarantee that the method is well defined for  one iteration, as we will show
in the following result.
\begin{lemma} \label{l:wdef}
Let $x \in \Omega$. If \,$\sigma(x)<r$,  then $\mathcal{G}_F$ is well
defined and there holds
\begin{multline}\label{oo}
  \|\mathcal{G}_{F}(x)-x_*\| \leq
    \frac{[f'(\sigma(x))+1+\kappa][f'(\sigma(x))\sigma(x)-f(\sigma(x))]}{[\sigma(x)f'(\sigma(x))]^2}{\|x-x_*\|}^{2}+\\ \frac{ (1+\sqrt{2}) c \beta [f'(\sigma(x))+1]^2}{[\sigma(x)f'(\sigma(x))]^2}{\|x_k-x_*\|}^{2}+
    \frac{c\beta[(1+\sqrt{2})\kappa+1][f'(\sigma(x))+1]}{\sigma(x)[f'(\sigma(x))]^2}\|x-x_*\|.
\end{multline}
In particular,
\begin{equation}\label{oooo}
\|\mathcal{G}_{F}(x)-x_{*}\|< \|x-x_*\|.
\end{equation}
\end{lemma}
\begin{proof}
First, as $\|x-x_*\|<r$, it follows from Lemma \ref{wdns} that
$H(x)=F'(x)^* F'(x)$ is invertible; then  $G_F(x)$ and
$\mathcal{G}_F(x)$ are well defined. Now, as
$-F'(x_*)^*F(x_*)\in
\partial J(x_*)$ and $F '(x_*)$ is injective, it follows from
Proposition~\ref{pro:a1}, \eqref{NFS} and \eqref{NF} that
$x_*=\mbox{prox}_{J}^{H(x_*)}(G_F(x_*)).$ Hence,
$$\|
\mathcal{G}_F(x)-x_{*}\|=\|\mbox{prox}_{J}^{H(x)}(G_F(x))-\mbox{prox}_{J}^{H(x_*)}(G_F(x_*))\|,
$$
which, combined with  Lemma \ref{lemma4} yields
\begin{multline*}
  \| \mathcal{G}_F(x)-x_{*}\|\leq (\|H(x)\|\|H(x)^{-1}\|)^{1/2}\|G_F(x)-G_F(x_*)\| \\+\|H(x)^{-1}\|\|(H(x)-H(x_*))(G_F(x_*)-
  \mbox{prox}_{J}^{H(x_*)}(G_F(x_*))\|.
\end{multline*}
Using $x_*=\mbox{prox}_{J}^{H(x_*)}(G_F(x_*))$ and
\eqref{NF}, the last inequality becomes
\begin{multline*}\label{0012}
  \| \mathcal{G}_F(x)-x_{*}\|\leq (\|H(x)\|\|H(x)^{-1}\|)^{1/2}\|G_F(x)-G_F(x_*)\| \\+ \|H(x)^{-1}\|\|(H(x)-H(x_*))F'(x_*)^{\dagger}\|\|F(x_*)\|.
\end{multline*}
For simplicity, the notation defines the following terms:
\begin{equation}\label{aaa}
A(x,x_*)=(\|H(x)\|\|H(x)^{-1}\|)^{1/2}\|G_F(x)-G_F(x_*)\|
\end{equation}
 and
\begin{equation}\label{bbb}
B(x,x_*)=\|H(x)^{-1}\|\|(H(x)-H(x_*))F'(x_*)^{\dagger}\|\|F(x_*)\|.
\end{equation}
So, from the  three latter inequalities we have
\begin{equation}\label{d1}
\|{G}_F(x)-x_{*}\|\leq A(x,x^*)+ B(x,x^*).
\end{equation}
Now, we will obtain upper bounds of $A(x,x^*)$ and $B(x,x^*)$.
First, some algebraic manipulations and definitions in
\eqref{eq:def.er} and \eqref{NF} yield
\begin{align*}
 \|{G}_F(x)-G_F(x_{*})\|&=\|F'(x)^{\dagger}[F'(x)(x-x_*)-F(x)+F(x_*)]
+(F'(x_*)^{\dagger}-F'(x)^{\dagger})F(x_*)\|.\\ &
\leq\|F'(x)^{\dagger}\|\|E_{F}(x,x_*)\|
+\|F'(x_*)^{\dagger}-F'(x)^{\dagger}\| \|F(x_*)\|.
\end{align*}
Combining  last inequality,  Lemmas~\ref{wdns} and \ref{pr:taylor}
and definition of $c$, we have
$$\|
{G}_F(x)-G_F(x_{*})\|=\frac{e_f(\sigma(x),0)}{-f(\sigma(x))}+\frac{\sqrt{2}c\beta[f'(\sigma(x))+1]}{-f'(\sigma(x))}.
$$
So, the definition in \eqref{aaa}, last inequality  and
Lemma~\ref{pr:H}{\bf i}-{\bf ii}  imply
\begin{equation}\label{c21}
A(x,x^*)\leq \frac{{f'(\sigma(x))+1+\kappa}}{[f(\sigma(x))]^2}\left({e_f(\sigma(x),0)+\sqrt{2}c\beta[f'(\sigma(x))+1]}\right).
\end{equation}
On the other hand, from definition in \eqref{bbb}, items {\bf ii} and {\bf iii} of Lemma~\ref{pr:H} we have
\begin{equation}\label{c212}
B(x,x_*)\leq \frac{c\beta}{[f'(\sigma(x))]^2}(f'(\sigma(x))+2+\kappa)(f'(\sigma(x))+1).
\end{equation}
Hence, \eqref{d1},  \eqref{c21} and \eqref{c212}  imply
\begin{multline*}
\|\mathcal{G}_{F}(x)-x_*\| \leq
    \frac{[f'(\sigma(x))+1+\kappa]e_f(\sigma(x),0)}{[f'(\sigma(x))]^2}+
    \frac{(1+\sqrt{2})c\beta\left[f'(\sigma(x))+1\right]^2}{[f'(\sigma(x))]^2}+\\
    \frac{c\beta\left[(1+\sqrt{2})\kappa+1\right]\left[f'(\sigma(x))+1\right]}{[f'(\sigma(x))]^2},
\end{multline*}
which, combined with \eqref{eq:def.erf}, {\bf h1} and simple manipulation   yields to \eqref{oo}.

To end the proof first, note that the right-hand side of  \eqref{oo} is equivalent to
$$\left[\frac{\left[f'(\sigma(x))+1+\kappa\right]\left[\sigma(x)f'(\sigma(x))-f(\sigma(x))+c
 \beta(1+\sqrt{2})(f'(\sigma(x))+1)\right]+c
 \beta\left[f'(\sigma(x))+1\right]}{\sigma(x)[f'(\sigma(x))]^2}\right] \sigma(x).$$
On the other hand, as $x\in B(x_*,r)/\{x_*\}$, i.e., $0<\sigma(x)<r\leq\rho$ we apply the Proposition~\ref{pr:incr102} with $t=\sigma(x)$ to conclude that the quantity in the bracket above is less than one. So, \eqref{oooo} follows.

\end{proof}
\begin{remark}
If the inequality in \eqref{Hyp:MH} holds only for $\tau=0$, then \eqref{oo} becomes
\begin{multline*}
  \|\mathcal{G}_{F}(x)-x_*\| \leq
    \frac{\left[f'(\sigma(x))+1+\kappa\right]\left[f'(\sigma(x))\sigma(x)+f(\sigma(x))+2\sigma(x)\right]}{[\sigma(x)f'(\sigma(x))]^2}{\|x-x_*\|}^{2}+\\ \frac{ (1+\sqrt{2}) c \beta \left[f'(\sigma(x))+1\right]^2}{[\sigma(x)f'(\sigma(x))]^2}{\|x_k-x_*\|}^{2}+
    \frac{c\beta[(1+\sqrt{2})\kappa+1]\left[f'(\sigma(x))+1\right]}{\sigma(x)[f'(\sigma(x))]^2}\|x-x_*\|.
\end{multline*}
\end{remark}
\subsection{Proof of {\bf Theorem \ref{th:nt}}} \label{sec:proof}
First of all, note that the equation in \eqref{eq:DNS}  together
\eqref{NFS} and \eqref{NF} imply that the sequence $\{x_k\}$ satisfies
\begin{equation} \label{GF}
x_{k+1}=\mathcal{G}_F(x_k),\qquad k=0,1,\ldots \,.
\end{equation}
\begin{proof}
Since $x_0\in B(x_*,r)/\{x_*\},$ i.e., $0<\sigma(x_0)<r,$ by a combination of
Lemma~\ref{wdns}, the last inequality in Lemma~\ref{l:wdef} and an induction argument it is easy to see that  $\{x_k\}$  is well defined and remains in $B(x_*,r)$.

Now, our goal is to show that $\{x_k\}$ converges to $x_*$.
As $\{x_k\}$ is well defined and  contained in  $B(x_*,r)$,
 \eqref{GF} and Lemma \ref{l:wdef} leads to
\begin{multline*}
\|x_{k+1}-x_*\| \leq
\frac{[f'(\sigma(x_k))+1+\kappa][f'(\sigma(x_k))\sigma(x_k)-f(\sigma(x_k))]}{[\sigma(x_k)f'(\sigma(x_k))]^2}{\|x_k-x_*\|}^{2}+\\ \frac{ (1+\sqrt{2}) c \beta [f'(\sigma(x_k))+1]^2}{[\sigma(x_k)f'(\sigma(x_k))]^2}{\|x_k-x_*\|}^{2}+
    \frac{c\beta[(1+\sqrt{2})\kappa+1][f'(\sigma(x_k))+1]}{\sigma(x_k)[f'(\sigma(x_k))]^2}\|x_k-x_*\|,
\end{multline*}
for all $ k=0,1,\ldots.$. Using again  \eqref{GF} and the second part of Lemma \ref{l:wdef},
it is easy to conclude that
\begin{equation} \label{eq:icq2}
\sigma(x)=\|x_{k}-x_*\|< \|x_0-x_*\|=\sigma(x_0), \qquad \;k=1, 2 \ldots.
\end{equation}
Hence, by combining the last two inequalities with the last part of
Proposition~\ref{pr:incr101}, we obtain that
\begin{multline*}
\|x_{k+1}-x_*\| \leq
\frac{[f'(\sigma(x_0))+1+\kappa][f'(\sigma(x_0))\sigma(x_0)-f(\sigma(x_0))]}{[\sigma(x_0)f'(\sigma(x_0))]^2}{\|x_k-x_*\|}^{2}+\\ \frac{ (1+\sqrt{2}) c \beta [f'(\sigma(x_0))+1]^2}{[\sigma(x_0)f'(\sigma(x_0))]^2}{\|x_k-x_*\|}^{2}+
    \frac{c\beta[(1+\sqrt{2})\kappa+1][f'(\sigma(x_0))+1]}{\sigma(x_0)[f'(\sigma(x_0))]^2}\|x_k-x_*\|,
\end{multline*}
for all $ k=0,1,\ldots$, which is the inequality \eqref{eq:q2}.
Now, combining last inequality with \eqref{eq:icq2} we obtain
\begin{multline*}
\|x_{k+1}-x_*\| \leq
\bigg[ \frac{\left[f'(\sigma(x_0))+1+\kappa\right]\left[f'(\sigma(x_0))\sigma(x_0)-f(\sigma(x_0))+(1+\sqrt{2}) c \beta (f'(\sigma(x_0))+1)\right]}{\sigma(x_0)[f'(\sigma(x_0))]^2}+\\
    \frac{c\beta[f'(\sigma(x_0))+1]}{\sigma(x_0)[f'(\sigma(x_0))]^2}\bigg]\|x_k-x_*\|,
\end{multline*}
for all  $k=0,1,\ldots$. Applying Proposition
\ref{pr:incr102} with $t=\sigma(x_0)$ it is straightforward to
conclude from the latter inequality that $\{\|x_{k}-x_*\|\}$
converges to zero. So, $\{x_k\}$ converges to $x_*$.
\end{proof}
\section{Special cases} \label{sec:ec}
In this section, we present two special cases of
Theorem~\ref{th:nt}. The convergence theorem for
proximal Gauss-Newton method under Lipschitz condition and Smale's
theorem on proximal Gauss-Newton for analytic functions are included.
\subsection{Convergence result for Lipschitz condition}
In this section we show a correspondent theorem for Theorem
\ref{th:nt} under Lipschitz condition,  instead of the general
assumption \eqref{Hyp:MH}.
\begin{theorem}\label{th:ntqnnm}
Let $\Omega\subseteq \banacha$ be an open set,
$J:\Omega\to
\mathbb{R}\cup\{+\infty\}$ a proper, convex and lower semicontinuous
functional and $F:{\Omega}\to
\banachb$ a continuously differentiable function such that $F'$  has a closed image in $\Omega$. Let $x_* \in \Omega,$ $R>0$ and
$$
c:=\|F(x_*)\|, \qquad \beta:=\|F'(x_*)^{\dagger}\|,\qquad \kappa:=
\beta\left\|F'(x_*)\right\|\qquad \delta:=\sup \left\{ t\in [0, R):
B(x_*, t)\subset\Omega \right\}.
$$
Suppose that $-F'(x_*)^*F(x_*)\in \partial J(x_*)$, $F '(x_*)$ is
injective and there exists a $L>0$ such that
\begin{equation}\label{h10}
h:=[(1+\sqrt{2})\kappa+1] \,c\, \beta L<1, \qquad
\beta\left\|F'(x)-F'(x_*+\tau (x-x_*))\right\| \leq
L(1-\tau)\sigma(x),
\end{equation}
where $x\in B(x_*, \delta)$, $ \tau \in [0,1]$ and
$\sigma(x)=\|x-x_*\|$.
 Let
$$ r :=\min \left\{\frac{4+\kappa+2c(1+\sqrt{2})\beta  L-\sqrt{(4+\kappa+2 c(1+\sqrt{2})\beta  L)^2-8(1-h)}}{2
L}, \,\delta \right\}.$$
 Then, the proximal Gauss-Newton method for
solving \eqref{eq:p}, with starting point $x_0\in B(x_*, r)/\{x_*\}$
$$
x_{k+1}=\mbox{prox}_{J}^{H(x_k)}\big({x_k}-F'(x_k)^{\dagger}F(x_k)\big),\qquad
k=0,1\ldots,
$$
 is well defined, the generated sequence $\{x_k\}$ is
contained in $B(x_*,r)$, converges  to $x_*$ and
\begin{equation}\label{h11}
    \|x_{k+1}-x_*\| \leq
    \frac{\kappa L+2c(1+\sqrt{2})\beta L^2+L^2\sigma(x_0)}{2[1- L\sigma(x_0)]^2}{\|x_k-x_*\|}^{2}+
     \frac{[(1+\sqrt{2})\kappa+1]c\beta L}{[1- L\sigma(x_0)]^2}\|x_k-x_*\|,
\end{equation}
 for all $k=0,1,\ldots.$
\end{theorem}
 \begin{proof}
It is immediate to prove that $F$, $x_*$ and $f:[0, \delta)\to
\mathbb{R}$ defined by $ f(t)=Lt^{2}/2-t, $ satisfy the inequality
\eqref{Hyp:MH}, conditions {\bf h1} and {\bf h2}. Since
$[(1+\sqrt{2})\kappa+1] \,c\, \beta  L<1$,
  the condition {\bf h3} also holds. In this case, it is easy to see that the
constants $\nu$  and $\rho$ as defined in Theorem~\ref{th:nt},
satisfy
$$0<\rho=\frac{4+\kappa+2c(1+\sqrt{2})\beta  L-\sqrt{(4+\kappa+2 c(1+\sqrt{2})\beta  L)^2-8(1-h)}}{2
L} \leq \nu=1/L ,$$ as a consequence, $ 0<r=\min \{\delta,\,\rho\}.
$ Therefore, as $F$, $J$, $r$, $f$ and $x_*$ satisfy all of the
hypotheses of  Theorem \ref{th:nt}, taking  $x_0\in B(x_*,
r)\backslash \{x_*\}$ the statements of the theorem follow from
Theorem~\ref{th:nt}.
\end{proof}
\begin{remark}
If the second inequality in \eqref{h10} holds only for $\tau=0$,
then an analogous theorem holds true. More specifically, if we
replace the definition of $r$ in above theorem by
$$ r :=\min \left\{\frac{-(4+3\kappa+2c(1+\sqrt{2})\beta  L)-\sqrt{(4+3\kappa+2 c(1+\sqrt{2})\beta L)^2+8(1-h)}}{2
L}, \,\delta \right\}.$$
 then all the statements of the previous theorem are valid with exception of inequality
\eqref{h11}, which in this case
 becomes
$$
    \|x_{k+1}-x_*\| \leq
    \frac{3\kappa L+2c(1+\sqrt{2})\beta L^2+3L^2\sigma(x_0)}{2[1- L\sigma(x_0)]^2}{\|x_k-x_*\|}^{2}+
     \frac{[(1+\sqrt{2})\kappa+1]c\beta L}{[1- L\sigma(x_0)]^2}\|x_k-x_*\|,
$$
 for all $k=0,1,\ldots.$
\end{remark}
For the zero-residual problems, i.e., $c=0$, the Theorem \ref{th:ntqnnm} becomes:
\begin{corollary}\label{cor:li1}
Let $\Omega\subseteq \banacha$ be an open set,
$J:\Omega\to
\mathbb{R}\cup\{+\infty\}$ a proper, convex and lower semicontinuous
functional and $F:{\Omega}\to
\banachb$ a continuously differentiable function such that $F'$  has a closed image in $\Omega$. Let $x_* \in \Omega,$ $R>0$ and
$$
\beta:=\|F'(x_*)^{\dagger}\|,\qquad \kappa:=
\beta\left\|F'(x_*)\right\|\qquad \delta:=\sup \left\{ t\in [0, R):
B(x_*, t)\subset\Omega \right\}.
$$
Suppose that  $F(x_*)=0$, $0 \in \partial J(x_*)$, $F '(x_*)$ is
injective and there exists a $L>0$ such that
$$
\beta\left\|F'(x)-F'(x_*+\tau (x-x_*))\right\| \leq
L(1-\tau)\sigma(x),
$$
where $x\in B(x_*, \delta)$, $ \tau \in [0,1]$ and
$\sigma(x)=\|x-x_*\|$. Let
$$ r :=\min \left\{\frac{4+\kappa-\sqrt{(4+\kappa)^2-8}}{2
L}, \,\delta \right\}.$$ Then, the proximal Gauss-Newton method for
solving \eqref{eq:p}, with starting point $x_0\in B(x_*, r)/\{x_*\}$
$$
x_{k+1}=\mbox{prox}_{J}^{H(x_k)}\big({x_k}-F'(x_k)^{\dagger}F(x_k)\big),\qquad
k=0,1\ldots,
$$
 is well defined, the generated sequence $\{x_k\}$ is
contained in $B(x_*,r)$, converges to $x_*$ and
$$
    \|x_{k+1}-x_*\| \leq
    \frac{\kappa L+L^2\sigma(x_0)}{2[1- L\sigma(x_0)]^2}{\|x_k-x_*\|}^{2},
$$
 for all $k=0,1,\ldots.$
 \end{corollary}
\subsection{Convergence result under Smale's condition }
In this section we present a correspondent theorem to Theorem
\ref{th:nt} under Smale's condition. For more details, see
\cite{MR1895083,MR1651750,S86}. First we will prove two auxiliary Lemmas.

\begin{lemma} \label{lemma:qc1}
Let $\Omega\subseteq \banacha$ be an open set,
$F:{\Omega}\to \banachb$ an analytic function and
\begin{equation} \label{eq:SmaleCond}
\gamma := \sup _{ n
> 1 }\beta\left\| \frac {F^{(n)}(x_*)}{n
!}\right\|^{1/(n-1)}<+\infty,
\end{equation} where $ \beta:=\|F'(x_*)^{\dagger}\|$.
 Suppose that
$x_*\in \Omega$  and  $B(x_{*},
1/\gamma) \subset \Omega$. Then, for all $x\in B(x_{*}, 1/\gamma)$
there holds
$$
\beta\|F''(x)\| \leqslant  (2\gamma)/( 1- \gamma \|x-x_*\|)^3.
$$
\end{lemma}
\begin{proof}
The proof follows the same pattern as the proof of Lemma~21 of
\cite{MAX3}.
\end{proof}

The next result provides a condition which is easier to check than
 \eqref{Hyp:MH},  for two-times continuously differentiable functions.

\begin{lemma} \label{lc}
Let $\Omega\subseteq \banacha$ be an open set, $x_*\in \Omega$  and
$F:{\Omega}\to \banachb$  be twice continuously differentiable on
$\Omega$. If there exists a  twice
continuously differentiable function \mbox{$f:[0,R)\to \mathbb {R}$} such that
 \begin{equation} \label{eq:lc2}
\beta\|F''(x)\|\leqslant f''(\|x-x_*\|),
\end{equation}
for all $x\in B(x_*,R)\cap \Omega$, then $F$ and $f$
satisfy \eqref{Hyp:MH}.
\end{lemma}
\begin{proof}
The proof follows the same pattern as the proof of Lemma~22 of
\cite{MAX3}.
\end{proof}

\begin{theorem}\label{theo:Smale}
Let $\Omega\subseteq \banacha$ be an open set,
$J:\Omega\to
\mathbb{R}\cup\{+\infty\}$ a proper, convex and lower semicontinuous
functional and $F:{\Omega}\to
\banachb$ an analytic function such that $F'$  has a closed image in~$\Omega$. Let $x_* \in \Omega,$ $R>0$ and
$$
c:=\|F(x_*)\|, \qquad \beta:=\|F'(x_*)^{\dagger}\|,\qquad \kappa:=
\beta\left\|F'(x_*)\right\|\qquad \delta:=\sup \left\{ t\in [0, R):
B(x_*, t)\subset\Omega \right\}.
$$
Suppose that $-F'(x_*)^*F(x_*)\in \partial J(x_*)$, $F '(x_*)$ is
injective and
$$ h=2\,c\,
\gamma\beta[(1+\sqrt{2})\kappa+1]
 <1,
$$
recall that $\gamma := \sup _{ n
> 1 }\beta\left\| \frac {F^{(n)}(x_*)}{n
!}\right\|^{1/(n-1)}<+\infty$.
Let the constants $a=\gamma c \beta$, $b=(1+\sqrt{2})\gamma
c \beta$,
\begin{equation}\label{rhoo}
\bar{\rho}:=\inf \bigg\{s \in({\sqrt{2}}/{2},1): p(s):=
-4s^4+(1-\kappa+a+b(\kappa-1))s^3+(3+\kappa+a+b(\kappa-1))s^2+(b-1)s+b<0\bigg\},
\end{equation}
$$ r :=\min  \left\{(1-\bar{\rho})/\gamma, \,\delta \right\}.$$
Then, the proximal Gauss-Newton method for solving \eqref{eq:p},
with starting point $x_0\in B(x_*, r)/\{x_*\}$
$$
x_{k+1}=\mbox{prox}_{J}^{H(x_k)}\big({x_k}-F'(x_k)^{\dagger}F(x_k)\big),\qquad
k=0,1\ldots,
$$
 is well defined, the generated sequence $\{x_k\}$ is
contained in $B(x_*,r)$, converges to $x_*$ and
\begin{multline} \label{eq:q233}
    \|x_{k+1}-x_*\| \leq
    \frac{1+(\kappa-1)(1-\gamma \sigma(x_0))^2}{[1-2(1-\gamma \sigma(x_0))^2]^2}{\|x_k-x_*\|}^{2}+
 \frac{(1+\sqrt{2})\beta c\gamma^2(2-\gamma\sigma(x_0))^2}{[1-2(1-\gamma \sigma(x_0))^2]^2}{\|x_k-x_*\|}^{2}+\\
 \frac{c\beta[(1+\sqrt{2})\kappa+1](2-\gamma\sigma(x_0))(1-\gamma\sigma(x_0))^2}{[1-2(1-\gamma \sigma(x_0))^2]^2}\|x_k-x_*\|,
\end{multline}
 for all $k=0,1,\ldots.$
 \end{theorem}
\begin{proof}  Consider the real
function $f:[0,1/\gamma) \to \mathbb{R}$ defined by
$$
f(t)=\frac{t}{1-\gamma t}-2t.
$$
It is straightforward to show that $f$ is analytic and that
$$
f(0)=0, \quad f'(t)=1/(1-\gamma t)^2-2, \quad f'(0)=-1, \quad
f''(t)=(2\gamma)/(1-\gamma t)^3, \quad f^{n}(0)=n!\,\gamma^{n-1},
$$
for $n\geq 2$. It follows from the last equalities that $f$
satisfies {\bf h1}  and  {\bf h2}. Since $h=2\gamma\,c\, \beta
[(1+\sqrt{2})\kappa+1]
 <1$,
  the condition {\bf h3} also holds. Now, as
$f''(t)=(2\gamma)/(1-\gamma t)^3$ combining Lemmas~\ref{lemma:qc1}
and \ref{lc},
 we conclude that $F$ and $f$ satisfy \eqref{Hyp:MH}
with $R=1/\gamma$. In this case,
$$\nu=(2-\sqrt{2})/2\gamma<1/\gamma.$$
Now, we will obtain the constant  $\rho$ as defined in Theorem
\ref{th:nt}. For simplicity, consider the following change of
variable
$$s=1-\gamma t.$$
Then, $t=(1-s)/\gamma.$ Moreover, if $t$ satisfies
$0<t<\nu=(2-\sqrt{2})2\gamma$, then  $\sqrt{2}/2<s<1$. Hence,
determine the constant $\rho$ as defined in Theorem \ref{th:nt}  is
equivalent to determine the constant $s$ such that
\[
\bar{\rho}=\inf \bigg\{s \in({\sqrt{2}}/{2},1): p(s)=
-4s^4+(1-\kappa+a+b(\kappa-1))s^3+(3+\kappa+a+b(\kappa-1))s^2+(b-1)s+b<0\bigg\},
\]
where $a=\gamma c \beta$ and $b=(1+\sqrt{2})\gamma c \beta$.  Thus,
taking in account the change of variable, we have
$\rho=(1-\bar{\rho})/\gamma$ and
$$ r =\min  \left\{(1-\bar{\rho})/\gamma, \,\delta \right\}.$$
Therefore, as $F$, $J$, $r$, $f$ and $x_*$ satisfy all hypothesis of
Theorem \ref{th:nt}, taking $x_0\in B(x_*, r)\backslash \{x_*\}$,
the statements of the theorem follow from Theorem \ref{th:nt}.\end{proof}
\begin{remark}\label{r1}
Fixed the numerical values of  $a$, $b$ and $\kappa$, as $p(1)=h-1<0$, it is easy to compute  $\bar{\rho}$, defined in \eqref{rhoo}. Moreover, as $f(t)={t}/{(1-\gamma t)}-2t$
is the majorant function, by
Proposition~\ref{pr:incr101},  it follows that $p$ is decreasing in
$({\sqrt{2}}/{2},1)$.
\end{remark}


For the zero-residual problems, i.e., $c=0$, the Theorem \ref{theo:Smale} becomes:
\begin{corollary}\label{cor:tt}
Let $\Omega\subseteq \banacha$ be an open set,
$J:\Omega\to
\mathbb{R}\cup\{+\infty\}$ a proper, convex and lower semicontinuous
functional and $F:{\Omega}\to
\banachb$ an analytic function such that $F'$  has a closed image in~$\Omega$. Let $x_* \in \Omega,$ $R>0$ and
$$
\beta:=\|F'(x_*)^{\dagger}\|,\qquad \kappa:=
\beta\left\|F'(x_*)\right\|\qquad \delta:=\sup \left\{ t\in [0, R):
B(x_*, t)\subset\Omega \right\}.
$$
Suppose that  $F(x_*)=0$, $0 \in \partial J(x_*)$, $F '(x_*)$ is
injective and
$$
\gamma := \sup _{ n
> 1 }\beta\left\| \frac {F^{(n)}(x_*)}{n
!}\right\|^{1/(n-1)}<+\infty.
$$
Let be given the positive constants
\[
\bar{\rho}:=\inf \bigg\{s \in({\sqrt{2}}/{2},1):
p(s):=-4s^3+(1-\kappa)s^2+(3+\kappa)s-1<0\bigg\},\quad r :=\min
\left\{(1-\bar{\rho})/\gamma, \,\delta \right\}.
\]
Then, the proximal Gauss-Newton method for solving \eqref{eq:p},
with starting point $x_0\in B(x_*, r)/\{x_*\}$
$$
x_{k+1}=\mbox{prox}_{J}^{H(x_k)}\big({x_k}-F'(x_k)^{\dagger}F(x_k)\big),\qquad
k=0,1\ldots,
$$
is well defined, the generated sequence $\{x_k\}$ is
contained in $B(x_*,r)$, converges to $x_*$ and
$$
    \|x_{k+1}-x_*\| \leq
    \frac{1+(\kappa-1)(1-\gamma \sigma(x_0))^2}{[1-2(1-\gamma
    \sigma(x_0))^2]^2}{\|x_k-x_*\|}^{2}, \qquad k=0,1,\ldots.
$$
\end{corollary}

\section{Final remark}
Under a majorant condition, we present a new local convergence
analysis of the proximal Gauss-Newton method for solving penalized
nonlinear least squares problem. It would also be interesting to
present a semi-local convergence analysis of the proximal
Gauss-Newton method, under a majorant condition, for the problem on
consideration. This local analysis will be performed in the future.


\def\cprime{$'$}

\end{document}